\documentclass[draft, 12pt, reqno]{amsart} 
\usepackage{amssymb,latexsym,amsfonts,verbatim,amscd,ifthen} 
\usepackage{color}
\usepackage{relsize}

\usepackage[utf8]{inputenc}
\usepackage{mathrsfs}
\newcommand{\C}{{\mathbb C}}

\newcommand{\Pp}{{\mathbb P}}
\newcommand{\n}{\parallel}
\numberwithin{equation}{section}

\theoremstyle{plain}

\newtheorem{theorem}[equation]{Theorem}
\newtheorem{lemma}[equation]{Lemma}

\newtheorem{corollary}[equation]{Corollary}

\newtheorem*{definition*}{Definition}

\newtheorem{remark1}[equation]{Remark}
\theoremstyle{definition}
\newtheorem{Claim}[equation]{Claim}
 
\newtheorem{definition}[equation]{Definition}

\begin{document}

\title[Spectral rigidity for twisted $S_\nu U(2)$]{Spectral rigidity for infinitesimal generators of representations of twisted $S_\nu U(2)$. }
\author{Michael I Stessin}
\address{Department of Mathematics and Statistics \\
University at Albany \\
Albany, NY 12222}
\email{mstessin@albany.edu}

\begin{abstract}
We prove spectral rigidity theorems for infinitesimal generators of representations of twisted $S_\nu U(2)$ and for representations of $\mathfrak s \mathfrak l (2)$. 
\end{abstract}
\maketitle

\section{\textbf{Introduction and statement of results}}

\vspace{.2cm}

Given a tuple of $N\times N$ matrices $A_1,\dots,A_n$ the \textit{determinantal manifold (determinantal hypersurface)} $\sigma(A_1,\dots,A_n)$ is defined by
\begin{eqnarray}&\sigma(A_1,\dots,A_n) \nonumber \\
& \label{determinantal} \\
&=\Big \{[x_1:\cdots :x_n]\in \C	\Pp^{n-1}: \ det\big(x_1A_1+\cdots +x_nA_n)=0\big) \Big\}. \nonumber
\end{eqnarray}

In \cite{Y} R. Yang introduced an analog of ``determinantal manifold" for an arbitrary tuple of bounded operators acting on a Hilbert space $H$ and called it \textit{projective joint spectrum} of the tuple. It is defined as follows:
\begin{definition} 
Let $A_1,...,A_n$ be operators acting on a Hilbert space $H$. The \textit{projective joint spectrum} of the tuple $A_1,...,A_n$  is 
\begin{eqnarray}
&\sigma(A_1,...,A_n) \nonumber \\
& \label{projective} \\
&=\{ [x_1,...,x_n]\in \C{\mathbb P}^{n-1}: x_1A_1+\cdots +x_nA_n \ \mbox{is not invertible} \}. \nonumber 
\end{eqnarray}
\end{definition}

If the dimension of $H$ is finite, \eqref{determinantal} and \eqref{projective} coincide, so in the finite dimensionaln case, and this is the case we concentrate on in the present paper,  we use the terminology ``determinantal manifold" and ``projective joint spectrum" interchangiably. 

Determinantal manifolds of a matrix pencils have been under scrutiny for more than a hundred years. Notably, the study of group determinants led Frobenius to laying out the foundation of representation theory. There is an extensive literature on the question when an algebraic manifold of codimension 1 in the projective space admits a determinantal representation. Without trying to give an exhaustive account of the references on this topic, we just mention \cite{D1}-\cite{D5}, \cite{KV}, \cite{V}, and also the monograph \cite{D} and references there.

To avoid trivial redundancies, it is common to assume that at least one of the operators in the tuple (usually the last one) is invertible, and, therefore, 
can be taken to be the identity (or rather -$I$). In what follows, we will always assume that this is the case and write $\sigma(A_1,...,A_{n-1})$ instead of $\sigma(A_1,...,A_{n-1},-I)$. Also, it was found to be useful to deal with the part of the  projective joint spectrum that lies in the chart $\{x_n\neq 0\}$. This part is called the \textit{proper projective joint spectrum} 
and is denoted by $\sigma_p(A_1,...,A_{n-1})$,
\begin{eqnarray*}
&\sigma_p(A_1,...,A_n)  \\
  &=\{ (x_1,...,x_n)\in \C^n:  x_1A_1+\cdots+x_nA_n-I \ \mbox{is not invertible} \} .   
\end{eqnarray*}

\vspace{.3cm}

\vspace{.2cm}

In the last decade projective joint spectra of operator and matrix tuples have been intensively investigated 
(see \cite{BCY}-\cite{CST},\cite{DY}-\cite{KV}, \cite{MQW}, \cite{S}-\cite{SYZ},\cite{Y}) from the following angle: what does the geometry of the projective joint spectrum tell us about the relations between operators in the tuple? In particular, \cite{GY}, \cite{CST}, \cite{PS1} established spectral characterization of representations of infinite dihedral group, non-special finite Coxeter groups, and some subgroups of the permutation group related to the Hadamard matrices of Fourier type. In \cite{GLX} a spectral aspect of representations of simple Lie algebras and certain representations of ${\mathfrak s}{\mathfrak l}(2)$ were investigated. 

Somewhat more general type of results, which we call \textit{spectral rigidity results} are statements of the inverse type: ``\textit{if counting multiplicities the projective joint spectra of two operator or matrix tuples are the same, then (perhaps under some mild conditions) the tuples are equivalent.}" These type of results were established for representations of Coxeter groups, \cite{CST}, \cite{S}, and for subgroups the permutation group associated with Hadamard matrices, \cite{PS1}. The goal of this paper is to start the investigation of spectral rigidity for representation of Lie algebras, or related objects.


To see the difference between spectral characterization of representations and spectral rigidity, let us look at ${\mathfrak s}{\mathfrak l}(2)$.
The Lie algebra ${\mathfrak s}{\mathfrak l}(2)$ is generated by 
\begin{equation}\label{generators_sl2}
e=\left [ \begin{array}{cc} 0 & 1\\ 0 & 0\end{array} \right ],  \  f=\left [ \begin{array}{cc} 0&0 \\ 1&0 \end{array} \right ], \  \ h=\left [ \begin{array}{cc} 1 & 0\\ 0& -1 \end{array} \right ]. \end{equation}
These generators satisfy the commutation relations
\begin{equation}\label{commutation} 
[h,e]=2e, \ [h,f]=-2f, \ [e,f]=h.
\end{equation}

It is well-known that for each $n\geq 2$ there is a unique up to an equivalence $n $-dimensional irreducible representation of $\mathfrak{sl}(2)$, generated by the  $n\times n$ matrices $E_n,F_n,H_n$ which as operators on $\C^n$ act in the standard basis $\zeta_0,...,\zeta_{n-1}$ in the following way:
\begin{eqnarray}
F_n\zeta_j=\zeta_{j+1}, \  j=0,\dots,n-2, \ F_n\zeta_{n-1}=0, \nonumber \\
E_n\zeta_0=0, \ E_n\zeta_j=j(n-j)\zeta_{j-1}, \ j=1,\dots, n-1, \label{representation_sl2}\\
H_n\zeta_j=(n-1-2j)\zeta_j, \ j-0,\dots, n-1. \nonumber
\end{eqnarray}
This obviously implies that the joint spectrum of the images of the generators of ${\mathfrak s}{\mathfrak l}(2)$ under a finite-dimensional representation determines the representation.

An explicit expression of the projective joint spectrum $\sigma(H_n,E_n.F_n)$ can be found in \cite{GLX}. In particular, for $n=3$ 
$$E_3=\left [ \begin{array}{rrr} 0 & 2 & 0 \\ 0 & 0 & 2 \\ 0 & 0 & 0 \end{array}\right ], \ F_3=\left [ \begin{array}{rrr} 0&0&0 \\ 1& 0 & 0\\ 0 & 1&0 \end{array} \right ], \ H_3=\left [ \begin{array}{rrr} 2& 0 & 0\\ 0 &0 &0\\0 & 0  &-2\end{array}  \right ], $$
and

$$\sigma(H_3,E_3,F_3)=\Big\{[x:y:z:t]\in \C\Pp^3: \ t(4x^2+4yz-t^2)=0\Big\}. $$ 
Now, let \large $\alpha, \beta, \gamma, \delta  \in\C$\normalsize \ \  satisfy  \ \ \large $\alpha \gamma = \beta \delta =2$,\normalsize
$$A_1=H_3, \ A_2=\left [ \begin{array}{rrr}0& \alpha & 0 \\0&0&0\\ 0 &\beta & 0 \end{array} \right ],  \ A_3=\left [\begin{array}{rrr} 0&0&0 \\ \gamma &0&\delta \\ 0 & 0 &0\end{array} \right ], \ . $$
Then $$\sigma(A_1,A_2,A_3)=\sigma(H_3,E_3,F_3),$$
but $$ [A_2,A_3]=\left [ \begin{array}{rrr} 2 & 0 & \alpha \delta \\ 0 & -4 &  0\\ \beta \gamma &0 & 2 \end{array}  \right ] \neq A_1,$$

that is the joint spectrum of the tuple does not determine commutation relations, and, therefore, is not enough for spectral rigidity. Of course, the main reason for the lack of rigidity is the generators' nilpotance. 

To remedy the situation we look at the twisted $S_\nu U(2)$ matrix pseudogroup introduced by Woronovicz \cite{W}. As $\nu \to 1$, infinitesimal generators of a finite dimensional representation of $S_\nu U(2)$ approach generators of a representation of ${\mathfrak s}{\mathfrak l}(2)$. As shown in \cite {W}, the infinitesimal generators determine a representation of $S_\nu U(2)$ up to an equivalence, and a Theorem of Woronowicz (Theorem \ref{woronowicz_generators} below) states that similar to the situation with ${\mathfrak s}{\mathfrak l}(2)$, for each $n$ there is only 1 up to an equivalence representation of $S_\nu U(2)$. The definition of $S_\nu U(2)$ via a $C^*$-algebra approach suggests including in the tuple not only images of generators, but also their adjoints, and this leads to spectral rigidity results for infinitesimal generators of $S_\nu U(2)$ obtained in this paper. The limit case $\nu=1$ turns into spectral rigidity for ${\mathfrak s}{\mathfrak l}(2)$. The basic definitions of $S_\nu U(2)$ and infinitesimal generators of it's representations are recalled in the next section. For $\nu \in  [-1,1]\setminus\{0\}$ we denote the infinitesimal generators of the (unique up to an equivalence) self-adjoint $n$-dimensional representation of $S_\nu U(2)$ by $E_{n,\nu},F_{n,\nu}$ and $H_{n,\nu}$. As operators on $\C^n$ they act in the standard orthonormal basis $e_0,...,e_{n-1}$ according to \eqref{generators_explicit} in Theorem \ref{woronowicz_generators}, namely

\begin{eqnarray}
&F_{n,\nu}e_k=-c_{k+1}(\nu)e_{k+1}, \ H_{n,\nu}e_k=\frac{\nu^2}{1-\nu^2}(\nu^{2(n-2k-1)}-1)e_k, \nonumber \\ \label{generators}\\ &E_{n,\nu}e_k=\nu c_k(\nu)e_{k-1}, \nonumber\end{eqnarray} 
where
\begin{equation}\label{c}
c_k(\nu)=\frac{\nu}{1-\nu^2}\Big[\big(\nu^{n-2k-1}-\nu^{n-1}\big)\big(\nu^{1-n}-\nu^{n-2k+1}\big)\Big]^{1/2}. \end{equation}
Remark that the difference in exponents for $\nu$ between the ones in \eqref{generators}, \eqref{c} and the corresponding expressions in Theorem \ref{woronowicz_generators} is caused by different indexing: in \eqref{generators} and \eqref{c} the dimension of the representation is $n$ and $n$ is an integer, while  in Theorem \ref{woronowicz_generators} it is $2n+1$ with $n$ being a half-integer. Also in \eqref{generators} and \eqref{c} indexes for basic vectors vary from 0 to $n-1$, and in Theorem \ref{woronowicz_generators} - from $-n$ to $n$.

\vspace{.3cm}

Our main results are the following.

\begin{theorem}\label{rigidity_for_S_nu1}

Let $\nu \in [-1,1]\setminus\{0\}$. Suppose that 
 $A_1,A_2,A_3$ are $n$-dimensional matrices such that
 
 \begin{itemize} 
 \item[(a)] $A_1$ is normal. 
\item[(b)] The following equalities of pairwise projective joint spectra take place:
\begin{eqnarray*}
&\sigma_p(A_1,A_2A_2^\ast)=\sigma_p(H_{n,\nu},E_{n,\nu}E_{n,\nu}^\ast);  \
\sigma_p(A_1,A_2^\ast A_2)=\sigma_p(H_{n,\nu},E_{n,\nu}^\ast E_{n,\nu}); \\ 
&\sigma_p(A_1,A_3A_3^\ast)=\sigma_p(H_{n,\nu},F_{n,\nu}F_{n,\nu}^\ast); \ \sigma_p(A_1,A_3^\ast A_3)=\sigma_p(H_{n,\nu},F_{n,\nu}^\ast F_{n,\nu});\\
&\sigma_p(A_1,A_2A_3)=\sigma_p(H_{n,\nu},E_{n,\nu}F_{n,\nu}).
\end{eqnarray*}
\end{itemize}
 Then $\Big(A_1,A_2,A_3\Big)$ is unitary equivalent to $\Big(H_{n,\nu},E_{n,\nu},F_{n,\nu}\Big)$.
	\end{theorem}
\vspace{.2cm}

The following Corollary is straightforward, since the coincidence of joint spectra of a tuple implies the coincidence of the pairwise joint spectra.

\begin{corollary}\label{rigidity_for_S_nu2}
 Let $\nu \in [-1,1]\setminus\{0\}$. Suppose that 
 $A_1,A_2,A_3$ are $n$-dimensional matrices such that
\begin{itemize}
\item[(a)] $A_1$ is normal. 
 \item[(b)]  
 \begin{eqnarray*}
& \sigma_p\Big(A_1,A_2A_2^\ast, A_2^*A_2,A_3A_3^*,A_3^*A_3,A_1A_3\Big) \\
& =\sigma_p\Big(H_{n,\nu},E_{n,\nu}E_{n,\nu}^\ast,E_{n,\nu}^\ast E_{n,\nu},F_{n,\nu}F_{n,\nu}^\ast, 
  F_{n,\nu}^\ast F_{n,\nu},  E_{n,\nu}F_{n,\nu}\Big)	
\end{eqnarray*}
\end{itemize}
Then $\Big(A_1,A_2,A_3\Big)$ is unitary equivalent to $\Big(E_{n,\nu},H_{n,\nu},F_{n,\nu)}\Big)$.
\end{corollary}

\vspace{.2cm}

As $\nu \to 1$, \
$c_k(\nu) \to \sqrt{k(n-k)} $, so that 
the infinitesimal generators $ E_{n,\nu}, F_{n,\nu}$ and $H_{n,\nu}$ converge to 
 \begin{eqnarray}
 &\tilde{F}_n=\left [ \begin{array}{cccc} 0& 0 & \cdots & 0 \\ -\sqrt{n-1} & 0 & \dots &0\\ 0& -\sqrt{2(n-1)}& \cdots &0 \\ \cdot & \cdot & \cdot & \cdot \\ 0 & \cdots & -\sqrt{n-1} & 0 \end{array}\right ], \nonumber \\ 
 &\tilde{H}_n=\left [ \begin{array}{cccc} 1-n & 0 & \cdots & 0\\ 0&3-n& \cdots &0\\ \cdot & \cdot & \cdot& \cdot \\ 0& \cdots & 0 & n-1\end{array}\right ]  \label{limit at nu=1}, \\
 &\tilde{E}_n=\left [ \begin{array}{ccccc} 0 & \sqrt{n-1} & 0 & \cdots &0 \\0 & 0 & \sqrt{2(n-2)} & \cdots &0\\ \cdot & \cdot & \cdot & \cdot & \cdot \\ 0 & 0 & \cdots &0 &\sqrt{n-1}\\ 0 & 0& 0& \cdots & 0\end{array}\right ]. \nonumber
\end{eqnarray}

\vspace{.2cm}

\noindent Even though the tuples $(\tilde{F}_n,\tilde{H}_n,\tilde{E}_n)$ and $(-F_n,-H_n,E_n)$ are equivalent, they are not unitary equivalent. Since the statements of Theorem \ref{rigidity_for_S_nu1} and Corollary \ref{rigidity_for_S_nu2} include adjoints of $\tilde{E}_n$ and $\tilde{F}_n$, a similar rigidity result for the representation \eqref{representation_sl2} does not follow automatically. Nevertheless, a result similar to the one of Theorem \ref{rigidity_for_S_nu1} is true.

\begin{theorem}\label{rigidity_for_sl2_1}
Suppose that $E_n,F_n$, and $H_n$ are given by \eqref{representation_sl2}
  and $A_1,A_2,A_3$ are $n$-dimensional matrices such that the following are true:
 \begin{itemize}
 \item[a)] $A_1$ is normal;
 \item[b)]	Either
 \begin{eqnarray*}
&\sigma_p(A_1,A_2A_2^\ast)=\sigma_p(H_n,E_nE^\ast_n); \  
\sigma_p(A_1,A_2^\ast A_2)=\sigma_p(H_n,E_n^\ast E_n); \\ 
&\sigma_p(A_1,A_3A_3^\ast)=\sigma_p(H_n,F_nF^\ast_n); \  
\sigma_p(A_1,A_2A_3)=\sigma_p(H_n,E_nF_n),
\end{eqnarray*}
or 
\small \begin{eqnarray*}
& \sigma\Big(A_1,  A_2^*A_2,A_2A_2^*,A_3A_3^\ast, A_2A_3\Big) \
 =\sigma\Big(H_n,E_n^\ast E_n, E_nE_n^\ast,F_n F_n^\ast,E_nF_n\Big).	
\end{eqnarray*}\normalsize

\end{itemize}
Then $(A_1,A_2,A_3)$ is unitary equivalent to $(H_n,E_n,F_n)$. 
 
\end{theorem}

\vspace{.2cm}

\section{\textbf{Background results}}\label{background}
\subsection{Twisted $S_\nu U(2)$ pseudogroups and infinitesimal generators of their representations}

As mentioned above, matrix pseudogroups $S_\nu U(2)$ were introduced by Woronowicz in \cite{W} as deformations of $SU(2)$. Here we recall necessary definitions and basic results regarding representations of $S_\nu U(2)$. The details might be found in \cite{W}. Please, note that in order to adjust the material to our presentation, we slightly changed the notation adopted in \cite{W}.

\vspace{.2cm}

Let $\nu \in [-1,1]$, and $A$ be a $C^*$-algebra generated by 2 elements, $\alpha$ and $\gamma$, satisfying the conditions:
\begin{eqnarray}
&\alpha^*\alpha +\gamma^* \gamma=I, \ \alpha \alpha^* + \nu^2\gamma^*\gamma=I \nonumber \\
& \label{commutations_Snu}\\ 
&\gamma^* \gamma=\gamma \gamma^*, \ \alpha \gamma= \nu \gamma \alpha, \ \alpha \gamma^*=\nu \gamma^* \alpha \nonumber.
\end{eqnarray}	

$A$ is considered to be the non-commutative algebra of continuous functions on $S_\nu U(2)$.  Let 
$$u_\nu=\left [\begin{array}{rr}\alpha & -\nu \gamma^* \\ \gamma & \alpha^* \end{array} \right ]\in M_2 \otimes A . $$
Then 
\begin{itemize}
\item[1.] The $*$-algebra $\mathcal A$ generated by elements of $u$ is dense in $A$.
\item[2)] $u_\nu$ is an invertible element of $M_2\otimes A$.
\item[3.] If $\nu \neq 0$, the structure of Hopf algebra is introduced on $A$ by 
$$\Phi: A\to A\otimes A, $$  
$$ (id\otimes \Phi)u_\nu= \left [\begin{array}{rr} \alpha \otimes \alpha -\nu \gamma^* \otimes\gamma & -\nu \big (\alpha \otimes \gamma^*+ \gamma^*\otimes \alpha^*\big ) \\ \gamma \otimes \alpha +\alpha^*\gamma &-\nu \gamma \otimes \gamma^* + \alpha^* \otimes \alpha^*  \end{array} \right ]$$

and
\begin{eqnarray*}
&k: {\mathcal A} \to {\mathcal A}  \\
& \forall   a\in {\mathcal A}, \ \ k(k(a^*)^*)=a; \ \ (id\otimes k)u_\nu=u_\nu ^{-1}. \ \end{eqnarray*}
\end{itemize}

\vspace{.2cm}

We use the notation \  $u\odot v$	\ for the product of $N\times N$ matrices with entrees in $*$-algebras $B$ and $B^\prime$ where the usual product of matrix elements is replaced by tensor product, that is
$$ u=\left [ u_{ij}\right ]_{i,j=1}^N, \ v=\left [ v_{ij}\right ]_{i,j=1}^N \Longrightarrow u\odot v=\left [w_{ij}\right ]_{i,j=1}^N, \ w_{ij}= \sum_{k=1}^N a_{ik}\otimes b_{kj}.$$
Let $V$ be a finite dimensional vector space. We say that 	$v$ is a representation of $S_\nu U(2)$ acting on $V$, if $v$ is an invertible element of the algebra $B(V)\otimes A$ and
$$(id\otimes \Phi)v=v\odot v\in B(V)\otimes A\otimes A, $$
where
$$(m_1\otimes  a)\odot (m_2\otimes  b)=(m_1m_2)\otimes a\otimes b.$$
In particular, $u_\nu$  determines a representation acting on $\C^2$ which is called \textit{the fundamental representation of $S_\nu U(2)$}.

\vspace{.2cm}

Further, let $\mathcal M$ be subalgebra with unity in $M_4$ consisting of all 4x4 matrices (with complex entries) having non-zero elements only in the first row and on the main diagonal. 	The following elements of ${\mathcal M}$
$$\alpha_m(\nu)=\left [ \begin{array}{rlll}1&0&1&0\\ 0&\nu^{-1}&0&0\\ 0&0&\nu^{-2}&0 \\ 0&0&0&\nu^{-1} \end{array} \right ], \ \gamma_m(\nu)=\left [ \begin{array}{rrrr} 0&0&0&-\nu \\ 0&0&0&0 \\0&0&0&0\\0&0&0&0\end{array}\right ] $$
$$ \alpha_m^*(\nu)=\left [\begin{array}{rrrr} 1&0&-\nu^2& 0 \\ 0&\nu &0 &0 \\ 0&0&\nu^2 &0\\ 0&0&0&\nu \end{array} \right ], \ \gamma_m^*(\nu)=\left [ \begin{array}{rlrr} 0&\nu^{-1}&0&0 \\ 0&0&0&0 \\ 0&0&0&0 \\ 0&0&0&0 \end{array} \right ]$$
satisfy the commutation relations \eqref{commutations_Snu}. Therefore, $\exists \ \ f:\mathcal A \to \mathcal M$	such that $$f(\alpha)=\alpha_m(\nu), \ f(\gamma)=\gamma_m(\nu), \ f(\alpha^*)=\alpha_m^*(\nu), \ f(\gamma^*)=\gamma_m^(\nu)*.$$
Denote by $e_0^\nu(a),\chi_0^\nu(a),\chi_1^\nu(a),\chi_2^\nu(a),f_0^\nu(a)f_1^\nu(a),f_2^\nu(a)$ linear fuctionals on $\mathcal A$ such that
$$f(a)=\left [\begin{array}{cccc} e_0^\nu(a) & \chi_0^\nu(a) &\chi_1^\nu(a) &\chi_2^\nu(a) \\ 0&f_0^\nu(a) & 0&0\\ 0&0&f_1^\nu(a) & 0 \\ 0&0&0&f_2^\nu(a)\end{array} \right ], \ a\in \mathcal A . $$

For a representation $v$ of $S_\nu U(2)$ acting on $V$, \ $v\in B(\mathcal V)\otimes \mathcal A$, therefore, for a linear functional $\psi$ on $A$ we may introduce the operator
$$(id\otimes \psi)v\in B(V) .$$

\begin{definition}
The operators:
\begin{eqnarray*}
E_\nu(v)=(id\otimes \chi_0^\nu)v, \ H_\nu(v)=(id \otimes \chi_1^\nu)v, \ F_\nu(v)=(id \otimes \chi_2^\nu)v	
\end{eqnarray*}
are called \textit{infinitesimal generators } of $v$.	
\end{definition}

For the fundamental representation $u$ of $S_\nu U(2)$ these generators are:
$$E(u)=\left [ \begin{array}{cc} 0&1 \\ 0& 0 \end{array}\right ], \  H(u)=\left [ \begin{array}{cc} 1&0 \\ 0& -\nu^2 \end{array}\right ], \ F(u)=\left [ \begin{array}{cc} 0&0\\-\nu &0 \end{array}\right ]. $$


When $\nu=1$	these matrices generate $\mathfrak s \mathfrak l (2)$, as they coincide with \eqref{generators_sl2} except for $F(u)=-f$.
\vspace{.2cm}
\begin{theorem}[Woronowicz \cite{W}]
Let $v$ and $v^\prime$ be finite dimensional representations of $S_\nu U(2)$. Then
\begin{itemize}
	\item Infinitesimal generators determine a finite dimensional representation of $S_\nu U(2)$. For equivalent representations $v$ and $v^\prime$ the sets of infinitesimal generators are equivalent.
	\item The infinitesimal genrators of $v$ satisfy the following commuting conditions:
	\begin{eqnarray}
	\nu F_\nu(v)E_\nu(v))-\frac{1}{\nu}E_\nu(v))F_\nu(v)=H_\nu(v) \nonumber \\
	\nu^2 H_\nu(v)E_\nu(v)-\frac{1}{\nu^2}E_\nu(v)H_\nu(v)=(1+\nu^2)E_\nu(v) \label{infinitesimal_commutations} \\
	\nu^2 F_\nu(v)H_\nu(v)-\frac{1}{\nu^2}H_\nu(v)F_\nu(v)=(1+\nu^2)F_\nu(v)	 \nonumber
	\end{eqnarray}
	\end{itemize}
If $V$ is Hilbert, and $v$ is self-adjoint, then
$$-\nu E_\nu(v)^*=F_\nu(v), \ \ H_\nu^*(v)=H_\nu(v). $$	
\end{theorem}
\vspace{.2cm}

Any tuple $(A_0,A_1,A_2)$ satisfying  commuting relations \eqref{infinitesimal_commutations} is called \textit{infinitesimal representation} of $S_\nu U(2)$.	 The following result claims that for each $n\in{\mathbb N}$ there exists only one, up to an equivalence, $n$-dimensional representation of $S_\nu U(2)$ and specifically describes it.
\begin{theorem}[Wornowicz \cite{W}]\label{woronowicz_generators}
Let $(A_0,A_1,A_2)$ be a finite dimensional irreducible infinitesimal representation of $S_\nu U(2)$. Then, the eigenvalues of $A_1$ are real and, denoting  \ $\lambda_{max}$ \ \ the maximal one, we have the following possibilities:
\begin{itemize}
\item[1)] $\lambda_{max}=-\frac{\nu^2}{1-\nu^2}.$ \ \ 
Then \ \ $dim (V)=1$, \ \ and $\exists \ c\in \C	$ such that
$$A_0=c\frac{\nu}{1-\nu^2}I, \ A_1=-\frac{\nu^2}{1-\nu^2}I, \ A_2=\frac{1}{c}\frac{\nu^2}{1-\nu^2}I. $$ 
\item[2)] $\lambda_{max}=\frac{\nu^2}{1-\nu^2}(\nu^{-4n}-1)$, where $n=0,1/2,1,3/2,..$ is a non-negative integer or half-integer. Then  \ $dim (V)=2n+1$ \ and  \  $\exists  \ \eta_{-n},...,\eta_n$ - a basis of $V$ such that
\begin{equation}\label{generators_explicit}
A_0\eta_k=-c_{k+1}\eta_{k+1}, \ A_1\eta_k=\frac{\nu^2}{1-\nu^2}(\nu^{-4k}-1)\eta_k, \ A_2\eta_k=\nu c_k\eta_{k-1}, 
\end{equation}
where $c_k=\frac{\nu}{1-\nu^2}\big[(\nu^{-2k}-\nu^{2n})(\nu^{-2n}-\nu^{2(k-1)})\big]^{1/2} $ and $\eta_{n+1}=0$..

\item[3)] If $V$ is Hilbert, and the representation is self-adjoint, 	the case 1) cannot happen, and the basis $\eta_{-n},...,\eta_{n}$ is orthogonal.
\end{itemize}	
\end{theorem}

\subsection{Joint spectra of commuting tuples} 
The following results could be found in \cite{CSZ}.

\begin{theorem}[Theorem A, \cite{CSZ}]\label{csz1}
Suppose $(A_1,\cdots,A_n)$ is a tuple of compact, self-adjoint operators acting on a Hilbert
space $H$. Then the operators in $(A_1,...,A_n)$ pairwise commute if and only if $\sigma_p(A_1,...,A_n)$
consists of countably many, locally finite, complex hyperplanes in $\C^n$.	
\end{theorem}

\vspace{.2cm}

\begin{theorem}[Theorem B, \cite{CSZ}]\label{csz2}
Suppose $(A_1,\cdots,A_n)$ is a tuple of $N\times N$ normal matrices. Then the 
following conditions are equivalent:
\begin{enumerate}
\item[(a)] The matrices in $(A_1,...,A_n)$ pairwise commute.
\item[(b)] $\sigma_p(A_1,...,A_n)$ is the union of finitely many complex hyperplanes in $\C^n$.
\item[(c)] The complex polynomial
$$p(z_1,\cdots,z_N)=\det(z_1A_1+\cdots+z_nA_n-I)$$
is completely reducible.
\end{enumerate}
	\end{theorem}

\vspace{.2cm}

\subsection{\textbf{Spectral compressions}}

In the proof of Theorems \ref{rigidity_for_sl2_1} and \ref{rigidity_for_S_nu1} we will need the following result which is an adjusted to our needs simplified version of the first statement in Theorem  1.15 in \cite{S}.

\vspace{.2cm}

Let $A_1$ be a normal matrix with eigenvalues $\lambda_1,\dots,\lambda_m$ and spectral resolution
$$A_1=\displaystyle \sum_{j=1}^m \lambda_j P_{\lambda_j}.$$
Spectral projections $P_j$ are given by
$$P_{\lambda_j}=\frac{1}{2\pi i}\int_\gamma (w-A_1)^{-1}dw, $$
where $\gamma$ is a conotur that separates $\lambda_j$ from the rest of the spectrum of $A_1$.

\begin{theorem}\label{compressions}
Let $A_1,A_2$ be $n\times n$ matrices with	$A_1$ being normal. Suppose that the line
$$\lambda x_1+\mu x_2=1 $$
is in $\sigma_p(A_1,A_2)$.  Then
\begin{itemize}
\item $\lambda$ is an eigenvalue of $A_1$.	
\item If
the multiplicity of this eigenvalue (and, therefore, of this line too) is equal to 1, then
\begin{equation}\label{compress}
P_\lambda A_2P_\lambda=\mu P_\lambda.	
\end{equation}
\end{itemize}
\end{theorem}

\section{\textbf{Proof of Theorem \ref{rigidity_for_S_nu1}}}\label{proof_for_S_nu}

\vspace{.2cm}

Fix an orthonormal eigenbasis $\zeta=(\zeta_0,...,\zeta_{n-1})$ for $A_1$, so that $A_1$ is diagonal in this basis and write $A_2$ and $A_3$ in this basis. 

\vspace{.3cm}


First we observe that 
condition (b) of Theorem \ref{rigidity_for_S_nu1} implies that
\begin{eqnarray}
&\sigma(A_1)=\sigma(H_{n,\nu})=\Big\{ \frac{\nu^2}{1-\nu^2}\big( \nu^{2(n-2k-1)}-1\big): \ k=0,...,n-1\Big\}, \label{sigma(A1)}\\
&\sigma(A_2A_2^\ast)=\sigma(A_2^\ast A_2)=\sigma(E_{n,\nu}E_{n,\nu}^\ast) \nonumber \\
&=\bigg\{\nu^2c_k(\nu)^2: \ k=0,\dots,n-1\bigg\}. \label{sigma(A2A2*)} 
\end{eqnarray}

In particular,  \eqref{sigma(A1)} and condition (a) of Theorem \ref{rigidity_for_S_nu1} imply that $A_1$ is self-adjoint and has a simple spectrum. 

Further,
it also follows from condition (b) of Thgeorem \ref{rigidity_for_S_nu1} that 
\begin{eqnarray}
&\sigma_p(A_1, A_2A_2^\ast)=\sigma_p(H_{n,\nu},E_{n,\nu}E_{n,\nu}^\ast)\nonumber \\
&=\displaystyle\prod_{j=0}^{n-1}\bigg (\frac{\nu^2}{1-\nu^2}\big( \nu^{2(n-2j-1)}-1\big)x_1-\nu^2c_{j+1}(\nu)^2x_2-1 \bigg), \label{ee*}\\
& \sigma_p(A_1, A_2^\ast A_2)=\sigma_p(H_{n,\nu}, E_{n,\nu}^\ast E_{n, \nu}) \nonumber \\
&=\displaystyle \prod_{j=0}^{n-1} \bigg (\frac{\nu^2}{1-\nu^2}\big( \nu^{2(n-2j-1)}-1\big)x_1-\nu^2c_j(\nu)^2x_2 -1\bigg), \label{e*e}
\end{eqnarray}
where in \eqref{ee*} - \eqref{e*e}  $c_0(\nu)=c_n(\nu)=0$, which agrees with \eqref{c}.

Thus, projective joint spectra \eqref{ee*} and \eqref{e*e} consist of projective lines. Since all three matrices $A_1$,  $A_2A_2^\ast$, and $A_2^\ast A_2$ are self-adjoint, Theorems \ref{csz1} and \ref{csz2} imply that $A_1$ commutes with both $A_2A_2^\ast$ and $A_2^\ast A_2$. Since $A_1$ has a simple spectrum, we conclude that  $\zeta$ is an eigenbasis for $A_2A_2^\ast$ and for $A_2^\ast A_2$, and, therefore, both $A_2A_2^\ast$ and $A_2^\ast A_2$ are diagonal matrices with
\begin{equation}\label{eigenvalues}
A_2A_2^\ast \zeta_j=\nu^2c_{j+1}(\nu)^2\zeta_j, \ A_2^\ast A_2\zeta_j=\nu^2c_j(\nu)^2\zeta_j, \ j=0,...,n-1.	
\end{equation}

Considering that for each $j$, the $j$-th diagonal entry of $A_2A_2^\ast$ is equal to the sum of squares of moduli of the entrees of the $j$-th row of $A_2$ and that the $j$-th diagonal entry of $A_2^\ast A_2$ is the sum of squares of moduli of the entries of  the $j$-th column of $A_2$, we conclude that the first column and the last row of $A_2$ are 0.

Write 
$$A_2=\left [ \begin{array}{cccc} 0 & b_{11} & \cdots & b_{1 \ n-1}\\
0 & b_{21}& \cdots & b_{2 \ n-1} \\ \cdot &\cdot & \cdot & \cdot \\0 & b_{n-1 \ 1} & \cdots & b_{n-1 \ n-1} \\ 0 & 0 & \cdots & 0\end{array} \right ]. $$
Since $A_2A_2^\ast$ is diagonal, the rows of the matrix $B=\big[ b_{ij}\big]_{i,j=1}^{n-1}$
and the columns of $B^\ast$ are orthogonal.
Also, \eqref{eigenvalues} shows that for every $1\leq j\leq n-1$
$$\displaystyle \sum_{m=1}^{n-1}|b_{j m}|^2=\nu^2c_{j+1}(\nu)^2, $$
and, therefore, 
$$B=DC, $$
where $D$ is the diagonal $(n-1)\times (n-1)$ matrix with $\nu c_{j+1}(\nu), \ j=0,...,n-2$ as the $j$-th entry on the main diagonal and $C=[c_{ij}]_{i,j=1}^{n-1}$ is a unitary $(n-1)\times (n-1)$ matrix. 

Of course,  this yields the following representation of $A_2$:
$$A_2=\tilde{D}\tilde{C}, $$
where $\tilde{D}$ is the diagonal $n\times n$ matrix obtained from $D$ by adding the $n$-th 0 column and the $n$-th 0 row, and $\tilde{C}$ is the $n\times n$ matrix obtained from $C$ by also adding the first 0 column and the $n$-th 0 row.

Since $\tilde{D}\tilde{D}^\ast=\tilde{D}^\ast \tilde{D}=A_2A_2^\ast$, we obtain
$$A_2^\ast A_2=\tilde{C}^\ast A_2A_2^\ast \tilde{C}.$$ 
Let us write 
$$U=\tilde{C}+W, $$
where the $n\times n$ matrix $W$ is ,
$$W=\left [ \begin{array}{cccc} 0 & 0 &\cdots & 0\\ \cdot & \cdot & \cdot & \cdot\\ 1 &  0 &\cdots & 0\end{array} \right ]. $$ 
Then $U$ is a unitary $n\times n$ matrix, and
\begin{eqnarray} 
&A_2=\tilde{D}(U-W),\label{A_2} \\
&A_2^\ast A_2=(U^\ast - W^\ast)A_2A_2^\ast (U-W). \label{A2*A2}
\end{eqnarray}
Since $\tilde{D}W=W^\ast A_2A_2^\ast =A_2A_2^\ast W=0$,  \eqref{A_2} and \eqref{A2*A2} yield

\begin{eqnarray} 
&A_2=\tilde{D}U  \label{A2 through U},\\
&A_2^\ast A_2=U^\ast A_2A_2^\ast U. \label{A2*A2 again}
\end{eqnarray}

By \eqref{eigenvalues} the .entrees of $A_2^\ast A_2$ on the main diagonal are obtained from the main diagonal entrees of $A_2A_2^\ast$ by the circular permutation
$${\mathscr P} = \left ( \begin{array}{cccc} 1 & 2 & \cdots & n \\ n & 1 & \cdots  & n-1   \end{array} \right ). $$
We denote the corresponding $n\times n$ matrix by the same letter ${\mathscr P}$, so that
$$ {\mathscr P}=\left [ \begin{array}{ccccc} 0& 1 &0 & \cdots & 0 \\0& 0&1 &\cdots &0 \\ \cdot & \cdot & \cdot & \cdot & \cdot \\ 0 & 0 & 0 &\cdots & 1 \\ 1 & 0 & 0 &\cdots &0\end{array} \right ] $$
and 
\begin{equation}\label{permutation}
A_2^\ast A_2= {\mathscr P}^\ast A_2A_2^\ast {\mathscr P}. 
\end{equation}

 Relations \eqref{A2*A2 again} and \eqref{permutation} and the fact that $U$ and ${\mathscr P}$ are unitary imply
\begin{equation}\label{commuting} 
U^\ast A_2A_2^\ast U={\mathscr P}^\ast A_2A_2^\ast {\mathscr P} \Longrightarrow   A_2A_2^\ast U {\mathscr P}^\ast=U {\mathscr P}^\ast A_2A_2^\ast .
\end{equation}

Next, we will show that the following Lemma holds.

\begin{lemma}\label{spectra}
For every $\nu\in [-1,1]\setminus \{0\}$ except for a finite number of points in this  punctured interval, $A_2A_2^\ast$ and $A_2^\ast A_2$ have simple spectra, and for each of these exceptional values of $\nu$ the multiplicity of every spectral point of these matrices does not exceed 2.
\end{lemma}


\noindent \begin{proof} 1. First, let $|\nu |<1$. By \eqref{sigma(A2A2*)} $A_2A_2^\ast$ and $A_2^\ast A_2$ have eigenvalues with multiplicities, if and only if for some $0\leq i\neq j\leq n-1$ we have $c_i(\nu)^2=c_j(\nu)^2$, which by \eqref{c} happens when
$$\big(\nu^{n-2i-1}-\nu^{n-1}\big)\big(\nu^{1-n}-\nu^{n-2i+1}\big)=\big(\nu^{n-2j-1}-\nu^{n-1}\big)\big(\nu^{1-n}-\nu^{n-2j+1}\big).  $$
This gives us
\begin{eqnarray*}
&\frac{1}{\nu ^{2i}}-1-\nu^{2(n-2i)}+\nu^{2(n-i)}=\frac{1}{\nu ^{2j}}-1-\nu^{2(n-2j)}+\nu^{2(n-j)}, \\
&\big(1-\nu^{2(n-i)}+\nu^{2n}\big)\nu^{2j}=\big(1-\nu^{2(n-j)}+\nu^{2n}\big)\nu^{2i}, \\
&\big( \nu^{2j}-\nu^{2i}\big)-\nu^{2n}\big(\frac{\nu^{2j}}{\nu^{2i}}-\frac{\nu^{2i}}{\nu^{2j}} \big) + \nu^{2n}\big( \nu^{2j}-\nu^{2i}\big)=0, \\
&\big(\nu^{2j}-\nu^{2i}\big)\big(1+\nu^{2n}-\nu^{2(n-j)}-\nu^{2(n-i)}\big)=0.
\end{eqnarray*}
Write $\nu^2=z$. Since $i\neq j$ the last equation implies
\begin{equation}\label{multiplicity}
1+z^n-z^{n-j}-z^{n-i}=0. 
\end{equation}
If $i+j\leq n, \ n-j\geq i$, and, therefore,
$$1+z^n-z^{n-j}-z^{n-i}\geq 1+z^n-z^i-z^{n-i}=\big(1-z^i\big)(1-z^{n-i}\big)>0. $$
Thus, if $c_i(\nu)^2=c_j(\nu)^2$ we mist have $i+j>n$. Further, suppose that $j>i$, so that $n-j<n-i$. Write \eqref{multiplicity} as
$$\big(1-z\big) \big(1+z+\cdots +z^{n-j-1}-z^{n-i}-\cdots -z^{n-1}\big)=0, $$
which implies
\begin{equation}\label{root}
1+z+\cdots +z^{n-j-1}-z^{n-i}-\cdots -z^{n-1}=0. 
\end{equation}
Since $i+j>n$, the number of negative coefficients of the polynomial in the left-hand side of \eqref{root} is bigger than the number of positive ones, so when $z$ runs from 0 to 1 this polynomial changes sign, and, therefore, has a root between 0 and 1. On the other hand, Descartes' rule of signs shows that this polynomial can have at most 1 positive root, and we conclude that for every pair $(i,j), \ i<j$ with $i+j>n$ there exists exactly 1 root \ $0<z_{ij}=\nu_{ij}^2<1$ such that $\big[c_i(\pm \nu_{ij})\big]^2=\big[c_j(\pm \nu_{ij})\big]^2$. Let us denote by $\tilde{S}$ the following finite subset of (-1,1):
\begin{equation}\label{tilde s}
\tilde{S}=\big\{ \pm \nu_{ij}=\pm\sqrt{z_{ij}}: \ i+j>n, \ i<j\big\}. 
\end{equation}
For every $\nu \in (-1,1)\setminus \big(\tilde{S}\cup 0\big)$ the matrices $A_2A_2^\ast$ and $A_2^\ast A_2$ have simple spectra. Also, suppose that $i<j<k$, $i+j>n, \ j+k>n$ and $z_{ij}=z_{ik}$, which means that $\nu_{ij}=\nu_{ik}$ and  $c_i(\nu_{ij})^2=c_j(\nu_{ij})^2 =c_k(\nu_{ij})^2$ is an eigenvalue of $A_2A_2^\ast$ of multiplicity greater than or equal to 3. 
Then it follows from \eqref{root} that
\begin{eqnarray*}
&1+z_{ij}+\cdots +z_{ij}^{n-k-1}-z_{ij}^{n-i}-\cdots -z_{ij}^{n-1}=0, \\
&1+z_{ij}+\cdots +z_{ij}^{n-k-1}+z_{ij}^{n-k}+\cdots +z_{ij}^{n-j-1} -z^{n-i}-\cdots -z^{n-1}=0,   	\end{eqnarray*}
a contradiction since $z_{ij}>0$. This shows that for $\nu \in \tilde{S}$ no eigenvalue of $A_2A_2^\ast$ has multiplicity higher than 2. 

\vspace{.1cm}

\noindent 2. Now,  let $|\nu|=1$. By \eqref{limit at nu=1} $c_i(\pm 1)^2=i(n-i)$, so it is easily seen that $c_i(\pm 1)^2=c_j(\pm 1)^2$ if and only if $i+j=n$.

This finishes the proof of Lemma \ref{spectra} \end{proof}

The following is a corollary to the proof of Lemma \ref{spectra}

\begin{corollary}\label{different i and j}
Suppose that $(i_1,j_1)$ and $(i_2,j_2)$ are such that $z_{i_1j_1}=z_{i_2j_2}$ and $i_1<i_2$. Then $j_1>j_2$ and $i_2-i_1>j_1-j_2$.	
\end{corollary}

\begin{proof} We have by \eqref{root}
\begin{eqnarray}
&1+z_{i_1j_1}+\cdots +z_{i_1j_1}^{n-j_1-1}-z_{i_1j_1}^{n-i_1}-\cdots -z_{i_1j_1}^{n-1}=0 \label{i_1j_1}\\
&1+z_{i_1j_1}+\cdots+z_{i_1j_1}^{n-j_2-1}-z_{i_1j_1}^{n-i_2}-\cdots -z_{i_1j_1}^{n-1}=0.\label{i_2j_2}
\end{eqnarray}
Since $i_2>i_1$, \ 
$$n-i_2<n-i_1 \ \Longrightarrow z_{i_1j_1}^{n-i_2}+\cdots +z_{i_1j_1}^{n-1}>z_{i_1j_1}^{n-i_1}+\cdots +z_{i_1j_1}^{n-1}\Longrightarrow j_2<j_1.$$ 
Also, \eqref{i_1j_1} and \eqref{i_2j_2} imply
$$z_{i_1j_1}^{n-j_1}+\cdots +z_{i_1j_1}^{n-j_2-1}=z_{i_1j_1}^{n-i_2}+\cdots +z_{i_1j_1}^{n-i_1-1}. $$
Since $0<z_{i_1j_1}<1$ and $n-j_2-1<n-i_2$ we see that 
$$(n-j_2-1)-(n-j_1-1)=j_1-j_2<(n-i_1-1)-(n-i_2-1)=i_2-i_1.$$
\end{proof}

Write $S=\tilde{S}\cup \big\{ \pm 1\big\}$, where $\tilde{S}$ is given by \eqref{tilde s}. Then for every $\nu \notin S$, \ $A_2A_2^\ast$ and $A_2^\ast A_2$ have simple spectra.

\vspace{.2cm}

\vspace{.2cm}

We now return tom the proof of Theorem \ref{rigidity_for_S_nu1}. There are  3 different cases.

\vspace{.1cm}

1). $\nu \notin S$.

Since in this case $A_2A_2^\ast$ is a diagonal matrix with distinct entries on the mail diagonal,
relation \eqref{commuting} shows that $U {\mathscr P}^\ast $ is a diagonal unitary matrix, so that
$$U{\mathscr P}^\ast=\Lambda=diag(e^{i\theta_1},\dots, e^{i\theta_n}). $$  
Hence,
\begin{equation}\label{U1} 
U=\Lambda {\mathscr P}.
\end{equation}
Now, \eqref{A2 through U} yields
\begin{equation} \label{A_2 again}
A_2 =\Lambda E_{n,\nu}.	
\end{equation} 

A similar argument that uses 
$$\sigma_p^d(A_1,A_3A_3^\ast)=\sigma_p^d(H_{n,\nu},F_{n,\nu}F_{n,\nu}^\ast);  \
\sigma_p^d(A_1,A_3^\ast A_3)=\sigma_p^d(H_{n,\nu},F_{n,\nu}^\ast F_{n,\nu}) $$ 
shows that for $\nu \notin S$ we have 
\begin{equation}\label{A_3 now}
A_3=F_{n,\nu}\Sigma^\ast,	
\end{equation}
where $\Sigma$ is a diagonal matrix with unimodular entries on the main diagonal. 

\noindent Further, \eqref{A_2 again}, \eqref{A_3 now} and relation $\sigma_p^d(A_1,A_2A_3)=\sigma_p^d(H_{n,\nu},E_{n,\nu}F_{n,\nu})$ from item b) in Theorem \ref{rigidity_for_S_nu1} imply that
$$ \Lambda=\Sigma,$$
and, therefore,
$$A_2=\tilde{\Lambda}E_{n,\nu}\tilde{\Lambda}^\ast, \ A_3=\tilde{\Lambda}E_{n,\nu}\tilde{\Lambda}^\ast, $$
where $\tilde{\Lambda}=diag \big(1,e^{-i\theta_1},e^{-i(\theta_1+\theta_2)},...,e^{-i(\theta_1+..\theta_{n-1})}\big)$. 
Since $\tilde{\Lambda}$ is diagonal, it commutes with $A_1$, which shows that
$$(A_1,A_2,A_3)=\tilde{\Lambda}(H_{n,\nu},E_{n,.\nu},F_{n,\nu})\tilde{\Lambda}^\ast, $$
This finishes the proof in the  case $\nu \notin S$.

\vspace{.2cm}

\noindent 2). $\nu\in S, \ |\nu|<1$. 

Let ${\mathscr M}(\nu)$ be the set of all pairs  $(ij)$ such that $\nu^2=z_{ij}$.  

In this case, \eqref{commuting} does not imply that $U{\mathscr P}^\ast$ is necessary a diagonal matrix, but rather the following. 

Let us denote by $Q_{ij}, \ i<j$ the $n\times n$ matrix corresponding to the exchange of the $i$-th and the $j$-th rows by the left multiplication. Then either $U{\mathscr P}^\ast$ is diagonal, or $\exists {\mathscr K}\subset {\mathscr M}(\nu)$ such that

\begin{equation}\label{U}
U=\Lambda \displaystyle \bigg ( \prod_{(i,j)\in {\mathscr K}}	Q_{i-1 \ j-1}\bigg ){\mathscr P}.
\end{equation}

Observe that 
\begin{eqnarray*}
&\sigma_p(A_1, A_2)=\sigma_p(H_{n,\nu},E_{n,\nu})=det\big(x_1H_{n,\nu}+x_2E_{n,\nu}-I\big) 
\end{eqnarray*} 
is a polynomial independent of $x_2$. 

\vspace{.2cm}

\begin{Claim}\label{claim}\textit{If ${\mathscr K}$ is not empty, then  \eqref{U} implies that the polynomial \ $$det \big( x_1A_1+x_2A_2-I\big)$$    contains non-trivial monomials with positive powers of $x_2$.}	
\end{Claim}

 Of course, this contradicts  $\sigma_p(A_1, A_2)=\sigma_p(H_{n,\nu},E_{n,\nu})$, so once the claim is proved, the present case is reduced to the previous one considered above.

\vspace{.2cm}

\noindent Suppose that ${\mathscr K}\neq \emptyset$. Write 
\begin{eqnarray}
&x_1A_1+x_2A_2-I=\big [ d_{ij}(x_1,x_2)\big ]_{i,j=0}^{n-1}, \label{determinant}\\ 
&det\big(x_1A_1+x_2A_2-I\big)=\displaystyle \sum_{{\mathcal P}} \displaystyle \prod_{l=0}^{n-1} (-1)^{sign(\mathcal P)}d_{l j_l}(x_1,x_2), \label{determinant_decomposition}
\end{eqnarray} 
where the sum is taken over all substitutions $\mathcal P=\left ( \begin{array}{ccc} 0 & \cdots & n-1\\ j_0 &\cdots & j_{n-1} \end{array} \right ) $. We also remark that each non-diagonal entry $d_{ij}(x_1,x_2), \ i\neq j$ is either 0, or a constant times $x_2$.

\vspace{.2cm}

Let ${\mathscr K}=\big\{(i_1,j_1),\dots, (i_k,j_k)\big\}, \ i_1<i_2<\cdots <i_k$. Corollary \ref{different i and j} implies $i_1<\cdots<i_k<j_k<\cdots<j_1$, and, as we saw above, $i_s+j_s>n$. 

First, let us consider the case when $j_k-i_k=1$. In this situation the entries on the main diagonal of $x_1A_1+x_2A_2-I$ are
\begin{eqnarray*}
&d_{ii}(x_1,x_2)=\frac{\nu^2}{1-\nu^2}\big(\nu^{2(n-2i-1)}-1\big)x_1-1, \  i\neq i_k, \\
&d_{i_ki_k}(x_1,x_2)=\frac{\nu^2}{1-\nu^2}\big(\nu^{2(n-2i-1)}-1\big)x_1+e^{i\theta_{i_k-1}}\nu c_{i_k}(\nu)x_2-1,
\end{eqnarray*}
and it follows that $det\big(x_1A_1+x_2A_2-I\big)$ contains the following terms with the first power of $x_2$
$$e^{i\theta_{i_k-1}}\nu c_{i_ki_k}(\nu)x_2\displaystyle \prod_{i\neq i_k} \Big[ \frac{\nu^2}{1-\nu^2}\big(\nu^{2(n-2i-1)}-1\big)x_1-1\Big].$$

\vspace{.2cm}
\noindent Now, suppose that $j_k-i_k>1$. In this case all entries on the main diagonal of $x_1A_1+x_2A_2-I$ are linear functions in $x_1$ and do not contain terms with $x_2$. Also, all non-trivial entries $d_{ij}(x_1,x_2)$ outside of the main diagonal are:
\begin{eqnarray}
&d_{ii+1}(x_1,x_2)=	e^{i\theta_i}\nu c_i(\nu)x_2, \ i=0,...,n-2, \ i\neq i_s-1, j_s-1, \nonumber \\ &s=1,...,k \nonumber \\
&d_{i_s-1 \ j_s}(x_1,x_2)=e^{i\theta_{j_s-1}} \ \nu c_{j_s}x_2, \ s=1,...,k \label{d_ij}\\
&d_{j_s-1 \ i_s}(x_1,x_2)=e^{i\theta_{i_s-1}} \nu c_{i_s}x_2, \ s=1,...,k. \nonumber
\end{eqnarray}
Since we have only $n-1$ non-zero entries outside of the main diagonal, each non-zero product in the rigfht-hand side of \ref{determinant_decomposition} contains an entry from the main diagonal.

\noindent Let us call a \textit{cycle} a set of pairs $\gamma=\big\{(m_1,m_2),(m_2,m_3),...,(m_r,m_{r+1})\big\}$ if $m_l\neq m_{l+1}$ and $.m_{r+1}=m_1$. We say that $r$ is \textit{the length of the cycle} and denote it by $\ell(\gamma)$. We say that a cycle is \textit{simple}, if it does not contain  proper subcycles, and  we call a cycle $\gamma$ \textit{non-trivial}, if all entries $d_{m_i,m_{i+1}}(x_1,x_2)$ do not vanish.

Given a substitution $\mathcal P=\left ( \begin{array}{ccc} 0 & \cdots & n-1\\ j_0 &\cdots & j_{n-1} \end{array} \right ) $, the set of pairs \newline
$(0,j_0),(1,j_1),\dots, (n-1,j_{n-1})$ is comprised of fixing pairs $(m,m)$ and simple cycles. It is easily seen from \eqref{d_ij} that each product in the right-hand side of \eqref{determinant_decomposition} is a non-trivial polynomial in $x_1$ times $x_2$ raised to the power equal to the sum of the lengths of simple cycles in $\mathcal P$.




Every entry of the first column and the last row of the matrix $x_1A_1+x_2A_2-I$ outside of the main diagonal is 0. Every other column and row contains only 1 non-zero entry outside of the main diagonal. Thus, every pair $(ij), \ i\neq j$ corresponding to a non-zero entry of $x_1A_1+x_2A_2-I$ may belong to only one non-trivial simple cycle.

It is easily seen that all non-trivial simple cycles are: 

\begin{eqnarray}
&\gamma_s=\Big\{(i_{s}-1,j_s),(j_s,j_s+1),\dots,(j_{s-1}-2,j_{s-1}-1),(j_{s-1}-1, i_{s-1})\Big\}, \nonumber \\&(i_{s-1},i_{s-1}+1),\dots, (i_s-2,i_s-1). \nonumber \\
 & \label{cycles}	\\
&\delta_s=\Big\{(j_s-1,i_s),(i_s,i_s+1),\dots, (i_{s+1}-2,i_{s+1}-1),(i_{s+1}-1,j_{s+1})\Big\}, \nonumber \\
&(j_{s+1},j_{s+1}+1),\dots, (j_s-2,j_s-1), \nonumber
 \end{eqnarray}
The corresponding lengths are:
\begin{eqnarray*}
&\ell(\gamma_s)=j_{s-1}-j_s +i_s-i_{s-1}, \\
&\ell(\delta_s)=j_s-j_{s+1}+i_{s+1}-i_s, \\
&\displaystyle \sum_s \Big(\ell(\gamma_s)+\ell(\delta_s)\Big)=j_1-i_1.
\end{eqnarray*}
Index pairs $ (i_1-1,j_1) \ \mbox{and} \ (i,i+1), \ \mbox{for} \ 0\leq  i\leq i_1-2 \ \mbox{and} \ j_1\leq i\leq n-2$ do not belong to any non-trivial cycle.
Thus, there is only one substitution $\mathcal P$ that contains all non-trivial cycles $\gamma_s$ and $\delta_s$ listed in \eqref{cycles} and fixing pairs $(i_1-1,i_1-1)$ and $(i,i)$ with $0\leq i\leq i_1-2 \ \mbox{and} \ j_1\leq i\leq n-1$. This shows that the polynomial $det(x_1A_1+x_2A_2-I)$ contains all monomials from
\begin{eqnarray*}
&\pm x_2^{j_1-i_1}e^{i(\theta_0+\cdots +\theta_{n-1})} \bigg(\displaystyle \prod_{m=i_1}^{j_1}\nu c_{m}(\nu)\bigg)\displaystyle \prod_{m=0}^{i_1-1} \bigg(\frac{\nu^2}{1-\nu^2}\big(\nu^{2(n-2m-1)}-1\big)x_1-1\bigg ) \\
&\times  \displaystyle \prod_{m=j_1}^{n-1}\bigg(\frac{\nu^2}{1-\nu^2}\big(\nu^{2(n-2m-1)}-1\big)x_1-1 \bigg),
\end{eqnarray*}
which proves Claim \ref{claim}.

\vspace{.2cm}

\noindent 3). Finally, let $|\nu|=1$.

\vspace{.1cm}

In this case ${\mathscr M}(1)={\mathscr M}(-1)=\Big \{(k,n-k): \ k=1,...,n-1\Big\}.$  The proof in this case is very similar to the one in the case $\nu\in S, \ |\nu|<1$: we prove the similar claim that ${\mathscr K}\neq \emptyset$ implies the existence of non-zero monomials with positive degrees of $x_2$, and this prove goes along practically the same lines.

\vspace{.1cm}

Theorem \ref{rigidity_for_S_nu1} is completely proved.

\section{\textbf{Proof of Theorem \ref{rigidity_for_sl2_1}  }}\label{proofs_for_sl2}

\vspace{.2cm}


Since 
$$\sigma_p(A_1,A_2A_2^\ast)=\sigma_p(H_n,E_nE_n^\ast),  \  \sigma_p(A_1,A_2^\ast A_2)=\sigma_p(H_n,E_n^\ast E_n),$$
and since $H_n$ has a simple spectrum,  $\sigma_p(A_1,A_2A_2^\ast )$, and $\sigma_p(A_1,A_2^\ast A_2)$ consist of complex lines, and spectral points of $A_2A_2^\ast$ and $A_2^\ast A_2$ have multiplicity 2 or 1, a proof similar to the one in Theorem \ref{rigidity_for_S_nu1} shows that a relation similar to \eqref{A_2 again}  holds. Namely,
\small \begin{equation}\label{still A_2} A_2=\Lambda E_n=\left [ \begin{array}{ccccc} 0 & e^{i\theta_0}(n-1) & 0 &\cdots & 0 \\ 0 & 0 & 2e^{i\theta_1}(n-2) & \cdots & 0 \\ \cdot & \cdot & \cdot & \cdot & \cdot \\ 0 & 0 & 0 & \cdots & e^{i\theta_{n-2}}(n-1) \\ 0 & 0 & 0 & \cdots & 0    \end{array} \right ],\end{equation}\normalsize

where, again, $\Lambda =diag(e^{i\theta_0},\dots,e^{i\theta_[n-1})$ is a diagonal matrix. 

\vspace{.2cm}

In our case 
\begin{equation}\label{F_nF_n*}
F_nF_n^\ast = diag \big( 0,\underbrace{1,\dots,1}_{n-1}\big), 
\end{equation}

so that, as above, by Theorem \ref{csz1} $A_1$ and $A_3A_3^\ast$ commute, and, since the spectrum of $A_1$ is simple, $A_3A_3^\ast$ is a diagonal matrix, and
relation \eqref{F_nF_n*} implies that  rows of $A_3$ from 1 to $n-1$ form an orthonormal system. In particular, if $A_3=\Big[ a^3_{ij}\Big]_{i,j=0}^{n-1} $, the Hilbert-Schmidt norm of $A_3$ satisfies
\begin{equation}\label{hs norm}
\parallel A_3\parallel_{HS}^2=trace(A_3A_3^\ast )=\displaystyle \sum_{i,j=0}^{n-1}|a^3_{ij}|^2=n-1.  
\end{equation}

\vspace{.2cm}

Further,
\begin{eqnarray}&\sigma_p(A_1,A_2A_3)=\sigma_p(H_n, E_nF_n)=\bigg \{ (x_1,x_2)\in \C^2: \nonumber \\
& \label{spectrum_A_2A_3} \\
&\Big( -(n-1)x_1-1\Big)\displaystyle \prod_{j=0}^{n-2} \Big((n-1-2j)x_1+(j+1)(n-j-1)x_2-1\Big)\bigg\}. \nonumber
 \end{eqnarray}
We will now apply Theorem \ref{compressions}. Denote by $P_j, \ j=0,...,n-1$ the projection \eqref{compress} corresponding to the eigenvalue $n-1-2j$ of $A_1$. By \eqref{still A_2} the compression 
$$P_jA_2A_3P_j=e^{i\theta_j}(j+1)(n-j-1)a^3_{j+1 j}P_j, \ j=0,...,n-2. $$ 
Theorem \ref{compressions} implies
$$e^{i\theta_j}(j+1)(n-1-j)a^3_{ j+1 j}=(j+1)(n-1-j),$$
so that 
\begin{equation}\label{A_3} a^3_{j+1 j}=e^{-i\theta_j}. \end{equation}
Now, \eqref{hs norm} yields
$$A_3=F_n\Lambda^\ast, $$
which finishes the proof of Theorem \ref{rigidity_for_sl2_1}.

\end{document}